%% file: main-revision-2.tex
\RequirePackage[ngerman,english]{babel}
\documentclass[12pt,a4paper]{amsart}

\include{head}

\graphicspath{{pics/}}
\begin{document}

\title{Computing the blocks of a quasi-median graph}

\author{Sven Herrmann \and Vincent Moulton}
\date{\today}
\email{sherrmann@mathematik.tu-darmstadt.de}
\email{Vincent.moulton@cmp.uea.ac.uk}
\address{School of Computing Sciences, University of East Anglia, Norwich, NR4 7TJ, UK}

\thanks{ email address and affiliation details of corresponding author: sherrmann@mathematik.tu-darmstadt.de\\ School of Computing Sciences, University of East Anglia, Norwich, UK\\ The first author was supported by a fellowship within the Postdoc"=Programme of the German Academic Exchange Service (DAAD)}

\begin{abstract}
Quasi-median graphs are a tool commonly used by 
evolutionary biologists to visualise the evolution of molecular 
sequences. As with any graph, a quasi-median graph
can contain cut vertices, that is, vertices whose
removal disconnect the graph. These
vertices induce a decomposition of the graph
into blocks, that is, maximal subgraphs which do 
not contain any cut vertices. Here we show that the
special structure of quasi-median graphs can be 
used to compute their blocks without having
to compute the whole graph.
In particular we present an algorithm 
that, for a collection of $n$ 
aligned sequences of length $m$, 
can compute the blocks of the associated quasi-median graph
together with the information required to correctly connect 	
these blocks together in run time $\bigO(n^2m^2)$, independent
of the size of the sequence alphabet.
Our primary motivation for presenting this algorithm is 
the fact that the quasi-median graph associated to a
sequence alignment must contain all most 
parsimonious trees for the alignment, and
therefore precomputing the blocks of
the graph has the potential to help
speed up \emph{any} method for computing such trees.
\end{abstract}


\maketitle

\section{Introduction}

Quasi-median graphs are a tool commonly used by 
evolutionary biologists to visualise the evolution of molecular 
sequences, especially mitochondrial sequences
(Schwarz and D"ur~\cite{SD11}; Ayling and Brown~\cite{AB08}; 
Bandelt et al.~\cite{BFSR95}; Huson et al.~\cite[Chapter~9]{HRS10}).
They were introduced by Mulder~\cite[Chapter~6]{M80} and their application to molecular sequence analysis was
introduced for binary sequences in (Bandelt et al.~\cite{BFSR95}) and 
for arbitrary sequences in (Bandelt et al.~\cite{BFR99}).
A quasi-median graph can be constructed for an alignment 
of sequences over any alphabet (Bandelt and D"ur~\cite{BD07}); 
for binary sequences they are 
also known as \emph{median graphs} (Bandelt et al.~\cite{BFSR95}). 
An example of a quasi-median graph associated to 
the hypothetical alignment of sequences $s_1$--$s_9$ 
is presented in Figure~\ref{fig:qm-graph} 
(see Bandelt and D"ur~\cite{BD07} for more details on how to 
construct such graphs).

\begin{figure}[ht]

{\tiny
\begin{tabular}{l|cccccccccccc}
&1&2&3&4&5&6&7&8&9&10&11&12\\
\hline
$s_1$&G&T&A&T&C&A&G&T&A&T&A&T\\
$s_2$&G&T&G&T&C&A&G&T&A&C&G&T\\
$s_3$&A&T&G&T&C&A&A&C&G&C&A&T\\
$s_4$&A&T&G&T&C&A&C&C&A&C&A&C\\
$s_5$&A&C&A&C&T&C&G&C&A&C&A&T\\
$s_6$&A&C&A&C&T&G&G&C&A&C&A&T\\
$s_7$&G&C&G&T&C&A&G&C&A&C&A&T\\
$s_8$&G&T&G&T&C&A&G&T&A&C&A&T\\
$s_9$&A&T&G&T&C&A&C&C&A&C&A&G\\
$s_{10}$&A&C&A&C&C&G&G&C&A&C&A&T
\end{tabular}\hfill\vspace{-3cm}}

\hfill\includegraphics[scale=0.75]{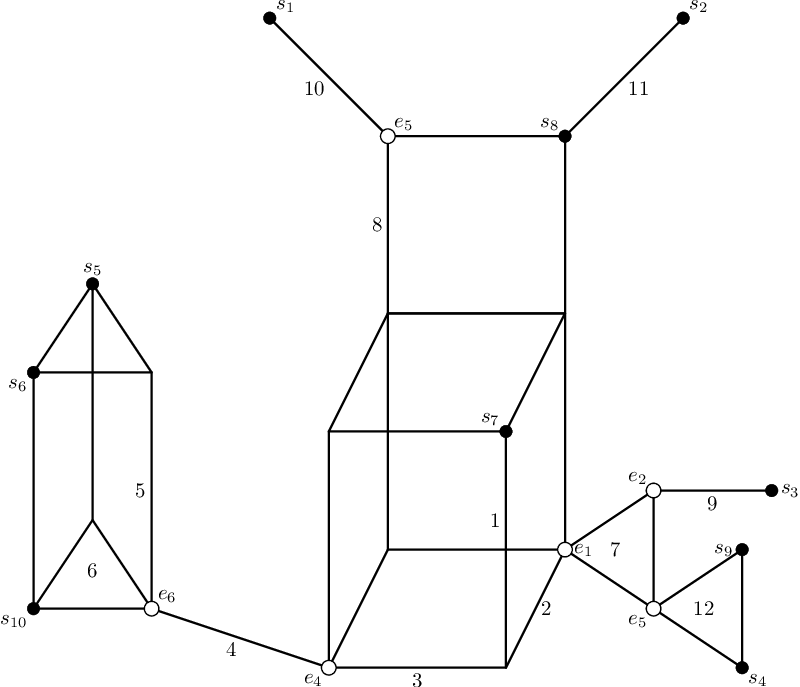}
  \caption{An alignment of hypothetical DNA sequences and
the associated quasi-median graph. The sequences
correspond to the black vertices and the 
columns correspond to the sets of edges, as indicated by the labels.}
  \label{fig:qm-graph}
\end{figure}

Here we are interested in computing the 
cut vertices of a quasi-median graph 
as well as an associated decomposition of the graph.
Recall that given a connected 
graph $G=(V(G),E(G))$, consisting of 
a set $V=V(G)$ of vertices and a set $E=E(G)$ of edges, 
a vertex $v\in V$ is called a \emph{cut vertex} of $G$ 
if the graph obtained by deleting $v$ and all edges 
in $E$ containing $v$ from $G$ is disconnected 
(for the basic concepts in graph theory that we use 
see, for example, (Diestel \cite{D10})). 
For example, in  quasi-median graph
in Figure~\ref{fig:qm-graph} the cut vertices are precisely the white vertices and the black vertex~$s_8$.
As with any graph, the cut vertices of a
quasi-median graph decompose it into {\em blocks}, that is, maximal 
subgraphs which do not contain any cut vertices themselves.
These blocks in turn, together with the information on how 
they are linked together, give rise to 
the \emph{block decomposition} of the graph
(see Section~\ref{sec:block_decomp} for a 
formal definition of this decomposition that we shall use 
which is specific to quasi-median graphs). It is well known, that the block decomposition of a given graph can be computed in linear time from its vertices and edges, however, the size of a quasi-median graph is usually exponential in the size of the sequence alignment.
Therefore, the main purpose of this paper is to provide
an algorithm for computing the 
block decomposition of a quasi-median graph {\em without}
having to compute the whole graph. 

The results in this paper complement the 
well-developed theory of quasi-median networks
(cf., e.g., (Bandelt et al. \cite{BMW94}; 
Imrich and Klav{\v{z}}ar \cite{IK00})). However, 
our primary motivation for computing 
the block decomposition of quasi-median graphs is provided 
by their close connection 
with most parsimonious trees (see, e.g., Felsenstein \cite{F04} for
an overview of parsimony). Indeed, Bandelt and R"ohl~\cite{BR09} 
showed that the set of {\em all} most parsimonious 
trees for a collection of (aligned, gap-free) sequences 
must be contained in the quasi-median graph of 
the sequences (see also (Bandelt~\cite{B01}) for a proof
of this result for median networks). More specifically,
they showed that the most parsimonious trees for the 
sequences are in one-to-one correspondence with 
the Steiner trees for the sequences considered 
as a subset of the vertices of the quasi-median graph.
It easily follows that the block decomposition of
a quasi-median graph can be used to break up
the computation of most parsimonious trees into
subcomputations on the blocks. Of course,
the quasi-median graph of an arbitrary
collection of sequences may not contain any cut vertices but,
as computing most parsimonious trees is NP-hard 
(Foulds and Graham \cite{FG82}),
it could still be a useful pre-processing step 
to compute the cut vertices of quasi-median graphs
before trying to compute most parsimonious trees. Similarly, Misra et al.~\cite{MBRS11} propose an integer linear programme for computing a most parsimonious tree, which is based on the structure of the quasi-median graph (called the \emph{generalised Buneman graph} by the authors). A computation of the block decomposition could be used to decompose the problem into smaller subproblems.

We now summarise the contents of the rest of this paper. 
We begin by presenting some preliminaries concerning 
quasi-median graphs in the next section. Then, in 
Section~\ref{sec:sc_sets}, we recall a 
characterisation of the vertices of
a quasi-median graph given in (Bandelt et al. \cite{MR1907821}), 
which we use in Section~\ref{sec:biject} to 
prove a key structural result for quasi-median graphs
(Theorem~\ref{thm:bijection}).
This result is 
a direct generalisation of Theorem~1 
of (Dress et al. \cite{MR1907821}) for median graphs, and 
states that the blocks in a quasi-median 
graph are in bijection with the connected components of 
a certain graph which can be associated to an alignment
that captures the degree of ``incompatibility''
between its columns. 
Using this result, we also derive a characterisation
of the cut vertices of a quasi-median graph (Theorem~\ref{characterize}).
After defining the block decomposition
of a quasi-median graph in Section~\ref{sec:block_decomp}, we 
present our algorithm for its computation in 
Section~\ref{sec:compute_block} (Algorithm~1).
In particular, we prove that this algorithm correctly
computes the  block decomposition (Theorem~\ref{thm:correct})
and also show that, for a collection of $n$ 
aligned sequences of length $m$, 
the algorithm's run time is $\bigO(n^2m^2)$, independent
of the size of the sequence alphabet (Theorem~\ref{thm:run_time}). We have implemented the algorithm and
it is available for download at
\url{http://www.uea.ac.uk/computing/quasidec}.

\textbf{Acknowledgments} The authors would like to thank the anonymous referees for their helpful comments, especially to one of them for pointing out the argument used in Lemma~\ref{lem:check_comp}. We would also like to thank Andreas Spillner for providing some useful observations concerning this argument.
\section{Preliminaries}

In the following we shall define
quasi-median networks in terms of partitions
rather than sequences, as explained in (Bandelt et al.~\cite{MR1907821}).
It is quite natural to do this since, given a 
multiple sequence alignment as in Figure~\ref{fig:qm-graph}, each 
column of the alignment 
gives rise to a partition of
the set of sequences in which all those sequences 
having the same nucleotide 
in the column are grouped together (note that 
columns with only one nucleotide are usually ignored).
In particular, by also recording the 
number of columns giving rise to a specific
partition, alignments can be recoded in terms
of sets of partitions of the sequences. 
This whole process is described in more 
detail in, for example, (Bandelt and D\"ur \cite{BD07}).

We now recall how quasi-median networks 
can be defined in terms of partitions. 
For the rest of this paper 
let $X$ denote an arbitrary, non-empty finite set. 
A {\em partition} $P$ of the set $X$ is a collection 
of non-empty subsets of $X$ 
whose union is $X$ and for which $A\cap B=\emptyset$ for 
all $A\not=B\in P$. For $x\in X$ we set $P(x)$ to be the 
unique element of $P$ that contains $x$.

\begin{exmp}\label{ex:one}
Consider the set $X=\{s_1,s_2,\dots,s_{10}\}$ of
sequences given in 
Figure~\ref{fig:qm-graph}. 
The columns labelled $1,\dots, 12$ give rise 
to the partitions 
$P_1,P_2,\dots, P_{12}$ of $X$, respectively. For example,
\[
P_1=\big\{\{s_1,s_2,s_7,s_8\},\{s_3,s_4,s_5,s_6,s_9,s_{10}\}\big\}\,,
\]
\[
P_4=\big\{\{s_1,s_2,s_4,s_5,s_7,s_8,s_9\},\{s_5,s_6,s_{10}\}\big\}\,,
\]
\[
P_6=\big\{\{s_1,s_2,s_3,s_4,s_7,s_8,s_9\},\{s_5\},\{s_6,s_{10}\}\big\}\,
\]
and the element of $P_7$ containing $s_6$ is given by
\[
P_7(s_6)=\{s_1,s_2,s_5,s_6,s_7,s_8,s_{10}\}.
\]
\end{exmp}

Let $\cP$ be an arbitrary set 
of partitions of $X$, also called {\em partition system on $X$}. 
A {\em $\cP$-map} 
is a map $v:\cP\to 2^X$ that maps every 
partition in $\cP$ to one of its parts. 
Note that, given any $x \in X$, the
map $v_x:\cP\to 2^X$ given by 
setting $v_x(P)=P(x)$ for $P \in \cP$ is a $\cP$-map.
In particular, we obtain a map $\pi: x\mapsto v_x$ from $X$ 
to the set of all possible $\cP$-maps. 

Now, given any three 
$\cP$-maps $v_1,v_2,v_3$, 
the \emph{quasi-median} $q(v_1,v_2,v_3)$ 
is defined to be the $\cP$-map 
\[
P\mapsto\begin{cases}
v_2(P),&\text{if }v_2(P)=v_3(P),\\
v_1(P),&\text{otherwise}
\end{cases}
\]
for $P \in \cP$.
The \emph{quasi-median hull} $H(\Phi)$ of a set 
$\Phi$ of $\cP$-maps is the smallest set of $\cP$-maps 
closed under taking quasi-medians, or, more formally, 
$H(\Phi)=\bigcup_{i\geq 0} H_i(\Phi)$, where
\[
H_0(\Phi)=\Phi\quad\text{and}\quad H_i(\Phi)
=\SetOf{q(v_1,v_2,v_3)}{v_1,v_2,v_3\in H_{i-1}(\Phi)}\,.
\]
The \emph{quasi-median graph} $Q(\cP)$ of a partition system $\cP$ 
on $X$ has vertex set $H(\pi(X))$ and 
edge set consisting of all those pairs $\{v_1,v_2\}$
of $\cP$-maps in $H(\pi(X))$ that differ on precisely one partition, 
that is, $\card{\smallSetOf{P\in\cP}{v_1(P)\not=v_2(P)}}=1$.
By this definition, for each edge $E=\{v_1,v_2\}$ of $Q(\cP)$, there exists precisely one partition $P(E)\in\cP$ with $v_1(P(E))\not=v_2(P(E))$ and, on the other hand, given a partition $P\in\cP$ we find an associated set $E(P)=\smallSetOf{\{v_1,v_2\}\in E(Q(\cP))}{v_1(P)\not=v_2(P)}$ of edges of $Q(\cP)$ \cite{MR1907821}. Removing all $E\in E(P)$ (without removing any vertices of $Q(\cP)$) yields a graph with $k:=\card P$ connected components $K_1,\dots,K_k$, such that
\[
\big\{\smallSetOf{x\in X}{v_x\in V(K_i)}\;|\;1 \leq i\leq k\big\}=P.
\]

\begin{exmp} 
The quasi-median graph of the partition system described in 
Example~\ref{ex:one} is depicted in Figure~\ref{fig:qm-graph}; the 
map $\pi$ gives the labelling of the 
black vertices in the graph by the sequences $s_1$ to $s_{10}$.
For example, the vertex $e_4$ maps 
partition $P_3$ to $\{s_1,s_5,s_6,s_{10}\}$, $P_4$ to 
$\{s_1,s_2,s_3,s_4,s_7,s_8,s_9\}$ and 
partition $P_6$ to $\{s_1,s_2,s_3,s_4,s_7,s_8,s_9\}$.

The vertex $e_6$ maps $P_4$ to $\{s_5,s_6,s_{10}\}$, showing, that $P(\{e_4,e_6\})=P_4$ and one can check that there is no other $E\in E(Q(\cP))$ with $P(E)=P_4$. So $E(P)=\{\{e_4,e_6\}\}$ and removing the edge $\{e_4,e_6\}$ from the graph gives to connected components corresponding to the two elements of $P_4$.
\end{exmp}

\section{Strong compatibility and quasi-median graphs}\label{sec:sc_sets}

We now consider a 
concept that is useful for understanding the
structure of quasi-median graphs 
(cf. (Bandelt et al. \cite{MR1907821})).
Two partitions $P,Q$ of $X$ are called 
\emph{strongly compatible} if either $P=Q$ or there 
exist $A\in P, B\in Q$ such that $A\cup B=X$ 
(see Dress et al.~\cite[p.3]{MR1453403}). Obviously, 
if distinct partitions $P,Q$ of $X$ are 
strongly compatible, then the 
sets $A$ and $B$ are necessarily unique;
we set $B(P,Q)=A$ and $B(Q,P)=B$. The following
observation concerning these sets will be useful later.

\begin{lem}\label{lem:Bequal}
Let $P,Q,R$ be distinct partitions of a set $X$ 
such that $P$ and $Q$ are not strongly compatible 
and $P,Q$ are both strongly compatible with $R$. 
Then $B(R,P)=B(R,Q)$.
\end{lem}
\begin{proof}
Since $R$ and $P$ are strongly compatible, we have 
$B(R,P)\cup B(P,R)=X$. If $B(R,Q)\not=B(R,P)$, this 
implies $B(R,Q)\subseteq B(P,R)$. So we get 
$B(Q,R)\cup B(P,R)\supseteq B(Q,R)\cup B(R,Q)=X$; 
a contradiction to $P$ and $Q$ not being strongly compatible.
\end{proof}

A partition system $\cP$ on $X$ is called 
\emph{strongly compatible} if each $P,Q\in \cP$ 
are strongly compatible. The following 
result, which is shown in the 
proof of Dress et al.~\cite[Lemma 3.1]{MR2726316}, will be useful 
later on for obtaining bounds on the number 
of cut vertices in a quasi-median graph.

\begin{prop}\label{prop:sc-cardinality}
Let $X$ be a set of cardinality $n\geq 2$ and $\cP$ 
be a strongly compatible set of partitions of $X$. 
Then $\card \cP\leq 3n-5$.
\end{prop}

We now consider a graph that will 
be key for our description of the 
block decomposition of a quasi-median graph.
The \emph{non-strong-compatibility graph} for
a partition system $\cP$ on $X$ 
(Bandelt and D"ur~\cite{BD07})) is the graph with 
vertex set $\cP$ and edge set 
\[
\smallSetOf{\{P,Q\}}{P\text{ and }Q\text{ are \emph{not} 
strongly compatible}}\,.
\] 
Properties of this graph have also been 
considered in (Schwarz and D"ur~\cite{SD11}).

\begin{exmp}
We continue Example~\ref{ex:one}. The non-strong-compatibility graph 
of the partition system is depicted in Figure~\ref{fig:ic-graph}. 
For example, the partitions $P_1$ and $P_5$ are strongly 
compatible with $B(P_1,P_5)=\{s_3,s_4,s_5,s_6,s_9,s_{10}\}$, 
$B(P_5,P_1)=\{s_1,s_2,s_3,s_4,s_7,s_8,s_9,s_{10}\}$. 
Similarly, $P_1$ and $P_6$ are not strongly compatible and -- 
as required by Lemma~\ref{lem:Bequal} -- $B(P_1,P_6)=B(P_1,P_5)$.
On the other hand, $P_3$ and $P_8$ are not strongly 
compatible, as we cannot find elements of the partitions 
whose union is $X$, which gives the edge $\{3,8\}$ 
in the non-strong-compatibility graph.
\end{exmp}

\begin{figure}[hbt]
  \includegraphics[scale=0.75]{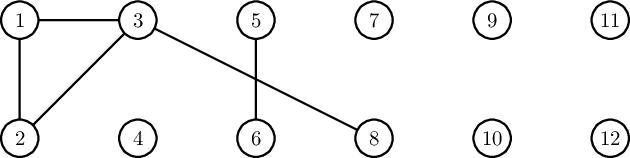}
  \caption{The non-strong-compatibility graph for the set of partitions 
           in Example~\ref{ex:one}. A vertex labelled $i$ corresponds to 
           partition $P_i$, $1 \le i \le 12$.}
  \label{fig:ic-graph}
\end{figure}

We now present some useful links between strong compatibility and
quasi-median graphs. The following result was
proved in (\cite[Theorem~1]{MR1907821}).

\begin{thm}\label{thm:qm-characterisation}
Let $\cP$ be a set of partitions of $X$. Then a 
$\cP$-map $\phi$ is a vertex of the quasi-median 
graph $Q(\cP)$ if and only if for every pair of 
distinct, strongly compatible partitions 
$P_1,P_2\in \cP$ either $\phi(P_1)=B(P_1,P_2)$ or 
$\phi(P_2)=B(P_2,P_1)$.
\end{thm}


Denote the complete graph on $n$ vertices by $K_n$,
and, for two graphs $G,H$, let $G\square H$
denote the (Cartesian) \emph{product} 
of $G$ and $H$, that is, the graph with vertex 
set $V(G)\times V(H)$ and edge set 
$\smallSetOf{\{(u,v),(u,w)\}}{\{v,w\}\in E(H)}
\cup \smallSetOf{\{(u,w),(v,w)\}}{\{u,v\}\in E(G)}$. 
In the extreme case of pairwise strong-compatibility
and non strong-compatibility
for a set of partitions, we have the 
following descriptions of the quasi-median graph 
(see Bandelt et al.~\cite[Theorem~2, Corollary~1]{MR1907821}).

\begin{thm}\label{thm:non-compatible}
Let $\cP$ be a set of partitions of $X$. Then
\begin{itemize}
\item[(i)] If every pair $P,Q\in\cP$ is
strongly compatible, then $Q(\cP)$ is a block graph, that is,
every block in $Q(\cP)$ is isomorphic to a complete graph.
\item[(ii)] If no distinct $P,Q\in\cP$ are strongly compatible,
then $Q(\cP)$ is isomorphic to $\square_{P\in\cP} K_{\card{P}}$.
\end{itemize}
\end{thm}

\section{Cut vertices and blocks in the quasi-median graphs}\label{sec:biject}

We now turn to understanding the cut-vertices and 
blocks of a quasi-median graph. 
By definition, for each edge $e=\{v_1,v_2\}$ 
of the quasi-median graph of a set of partitions $\cP$ 
of $X$, there exists exactly 
one $P\in\cP$ such that $v_1(P)\not=v_2(P)$. 
We say that $P$ is the {\em partition corresponding 
to $e$}.  Given a block $B$ of $Q(\cP)$ we denote by  
$\cP(B)$ the set of all $P\in\cP$ that correspond to 
some edge of $B$. The following result that 
relates the connected components of 
the non-strong-compatibility graph of $\cP$
with the blocks of $Q(\cP)$ is the key component 
to all that follows. Note that it
has been proved in the special case where all 
partitions in $\cP(B)$ have cardinality two in Dress et al.~\cite{DHKM11}.

\begin{thm}\label{thm:bijection}
Let $X$ be a finite set and $\cP$ be a 
partition system on $X$. Then the blocks of the 
quasi-median graph of $\cP$ are in bijection 
with the connected components of the non-strong-compatibility 
graph of $\cP$. More specifically, a bijection 
is given by mapping each block $B$ of the 
quasi-median graph  $Q(\cP)$ to the (necessarily) connected component 
of the non-strong-compatibility graph 
whose vertex set equals $\cP(B)$.
\end{thm}
\begin{proof}
We prove the theorem by induction on $\card \cP$. In the base case $\card \cP=1$, it follows from Theorem~\ref{thm:non-compatible} that $Q(\{P\})$ is isomorphic to a complete graph with $\card P$ vertices; the non-strong-compatibility graph of $\{P\}$ is just an isolated vertex.

Now let $\card \cP >1$ and choose some $P\in \cP$ and set $\cP'=\cP\setminus\{P\}$. 
By the induction hypothesis, the blocks of $Q(\cP')$ are 
in bijection with the connected components of the 
non-strong-compatibility graph of $\cP'$. First suppose 
that $P$ is strongly compatible to all $P'\in \cP'$. 
Obviously, the non-strong-compatibility graph of $\cP$ is 
derived from the non-strong-compatibility graph of $\cP'$ 
by adding the isolated vertex $P$. By 
Theorem~\ref{thm:qm-characterisation}, the vertices 
of $Q(\cP)$ are either just vertices of the subgraph 
isomorphic to $Q(\cP')$ or those $\cP$-maps $v$ defined by
\[
v(Q)=\begin{cases}
B(Q,P),&\text{if }Q\in\cP',\\
A,&\text{otherwise,}
\end{cases}
\]
for some $A\in P$. There can be only one vertex 
which is of both types, and this is the cut vertex 
separating the two types of vertices and hence the 
new block where all edges correspond to $P$ from 
the other blocks. The existence of the bijection 
now follows from the induction hypothesis.

Now suppose $P$ is not strongly compatible to some $Q\in\cP'$. For the non-strong-compatibility graph of $\cP$, this means the new vertex $P$ will be part of a new connected component that is the union of $\{P\}$ with all connected components of the non-strong-compatibility graph of $\cP'$ that contains some $Q\in\cP'$ not strongly compatible to $P$. By the induction hypothesis, these connected components are in bijection with blocks of $Q(\cP')$ and it follows that all those blocks are combined to create a new block $B$ of $Q(\cP)$. What remains to be shown is that there does not exist a block $B'$ of $Q(\cP')$ that is joined to $B$ and that does not contain any edge corresponding to some $Q\in\cP'$ that is not strongly compatible to $P$.
Suppose that would be the case. This would imply the existence of an edge $E$ of $B'$ that is on the shortest path between elements of two other blocks $B_1,B_2$ of $Q(\cP')$ that are joined into $B$. Let $P(E)=Q$ and let $R$ be in the block $B_1$ with $P$ and $R$ not strongly compatible and, similarly, $S$ in the block $B_2$ with $P$ and $S$ not strongly compatible. As $P$ and $R$ are not strongly compatible and, by assumption, $Q$ is strongly compatible to both $P$ and $R$, it follows from Lemma~\ref{lem:Bequal} that $B(Q,P)=B(Q,R)$. Similarly, we get $B(Q,P)=B(Q,S)$ which implies $B(Q,R)=B(Q,S)$. This, however, implies that the edges corresponding to $R$ and $S$ lie in the same connected component of the graph obtained from $Q(\cP')$ by removing $E(Q)$, contradicting the fact that $E$ is on the shortest path between elements of $B_1$ and $B_2$.
\end{proof}

\begin{exmp}
Considering Example~\ref{ex:one}, we see 
that the non-strong-compatibility graph in 
Figure~\ref{fig:ic-graph} has eight connected 
components: One whose vertex set consists of
the partitions $P_1,P_2$, $P_3$ and $P_8$, one containing
the partitions $P_5$ and $P_6$, and six isolated vertices
corresponding to the 
remaining partitions. This is in accordance to 
the eight blocks of the quasi-median graph in 
Figure~\ref{fig:qm-graph}, these being the large 
block in the middle of the graph, corresponding 
to $P_1$, $P_2$, $P_3$ and $P_8$, the block on the left
isomorphic to the Cartesian product of an edge and a triangle,
corresponding to $P_5$ and $P_6$,
two triangular blocks 
corresponding to the partitions $P_7$ and $P_{12}$ 
each having three parts, and 
five edges corresponding to partitions $P_4$, $P_9$,
$P_{10}$ and $P_{11}$ each having two parts.
\end{exmp}

It follows from Theorem~\ref{thm:bijection} that the 
collection of sets
$\cP(B)$ over all blocks $B$ of $Q(\cP)$
defines a partition $\Part(\cP)$ of $\cP$, and
that the following result holds that will be useful later.

\begin{cor}\label{cor:partitions}
Let $\cP$ be a partition system of $X$ with $|\cP| >1$, 
$P\in \cP$, $\cP':=\cP\setminus\{P\}$ and 
$I(\cP',P):=\smallSetOf{Q\in\cP'}{Q \text{ not 
strongly compatible to }P}$. Then we have
\[
\Part(\cP)=\SetOf{\cR\in \Part(\cP')}{I(\cR,P)=\emptyset}
\cup \left\{\bigcup\SetOf{\cR\in \Part(\cP')}{I(\cR,P)\not=\emptyset}
\cup\{P\}\right\}\,.
\]
In particular, if $I(\cP',P)=\emptyset$, we have 
$\Part(\cP)=\Part(\cP')\cup\{\{P\}\}$.
\end{cor}

Also, by Theorem~\ref{thm:bijection}
and Proposition~\ref{prop:sc-cardinality},
the following bounds on the number of 
cut vertices and blocks in a quasi-median graph must hold;
this will be useful for establishing run time
bounds for our main algorithm.

\begin{cor}\label{prop:num-cut vertices}
Let $X$ be a set of cardinality $n\geq 2$ and $\cP$ be a 
set of partitions of $X$. Then $Q(\cP)$ has at most 
$3n-5$ blocks and at most $3n-6$ cut vertices.
\end{cor}

We conclude this section by presenting a characterisation
for the cut vertices in a quasi-median graph
that is of independent interest, and will not 
be used later. First we prove a useful observation.

\begin{lem}\label{lem:help}
Let $\cP$ be a partition system of $X$ and
$v$ a cut vertex of $Q(\cP)$.
Suppose that $P_1,P_2 \in \cP$ are distinct 
and that $P_i$ corresponds to an edge
in the subgraph induced by $Q(\cP)$ on the
set $V(C_i)\cup\{v\}$, $i=1,2$, where $C_1, C_2$ 
are two distinct connected components of 
the graph $Q(\cP)$ with $v$ removed.
Then $P_1, P_2$ are strongly compatible, and
$v(P_1)=B(P_1,P_2)$, $v(P_2)=B(P_2,P_1)$ both hold.
\end{lem}
\begin{proof}
Since $P_1, P_2$ must be contained in distinct blocks
of $Q(\cP)$, it immediately follows by Theorem~\ref{thm:bijection}
that $P_1$ and $P_2$ are strongly compatible.

Now, by Theorem~\ref{thm:qm-characterisation} 
we can assume without loss of
generality that $v(P_1)=B(P_1,P_2)$. Let
$\{w,w'\}$ be an edge in $Q(\cP)$ that
corresponds to $P_1$. Without loss of generality, we 
can assume that there is path in $Q(\cP)$ from $w$
to $v$ such that no edge in this path 
corresponds to $P_1$ or $P_2$. In particular, we
have $w(P_1)=v(P_1)=B(P_1,P_2)$. Moreover,
$w'(P_1)\neq B(P_1,P_2)$ and so by 
Theorem~\ref{thm:qm-characterisation} $w'(P_2)=B(P_2,P_1)$.
But, $w'(P_2)=v(P_2)$ as, by 
Theorem~\ref{thm:bijection}, the block containing
all edges corresponding to $P_2$ must be contained
in the subgraph induced by $Q(\cP)$ on the
set $V(C_2)\cup\{v\}$. This completes the proof of the lemma.
\end{proof}

We now present the aforementioned characterisation of cut vertices.
Note that it generalises a characterisation
of cut vertices in median graphs given in Dress et al.~\cite{DHKM11}.

\begin{thm}\label{characterize}
Let $\cP$ be a partition system of $X$
and $v$ be a vertex of $Q(\cP)$. Then
$v$ is a cut vertex of $Q(\cP)$ if and only if the 
graph $G_{v}$ with vertex set $\cP$ and edge set 
$\smallSetOf{\{P,Q\}}{P, Q \in \cP, P\neq Q \mbox{ \em and }  
v(P) \cup v(Q) \neq X }$
is disconnected.
\end{thm}
\begin{proof}
Suppose that $v$ is a cut vertex of $Q(\cP)$. Then it follows
immediately by Theorem~\ref{thm:bijection} and
Lemma~\ref{lem:help} that $G_{v}$ is disconnected.
 
Conversely, suppose that $G_{v}$ is disconnected,
and, for contradiction, that $v$ is not a cut vertex of $Q(\cP)$.
Note that the non-strong compatibility graph
of $\cP$ is a subgraph of $G_{v}$.
Hence the non-strong compatibility graph
of $\cP$ is disconnected. Therefore, by  Theorem~\ref{thm:bijection}
there are at least two blocks in $Q(\cP)$. 

Now, suppose $B$ is the block of $Q(\cP)$ containing $v$.
By Theorem~\ref{thm:bijection} there must exist some 
block $B' \neq B$ of $Q(\cP)$ such that $\cP(B')$ is
contained in the vertex set of some connected component 
of $G_{v}$ that is not equal to the connected component
of $G_{v}$ whose vertex contains $\cP(B)$. Let
$w$ be the cut vertex of $Q(\cP)$ contained in $B$
which lies on a shortest path from $v$ to some vertex in $B'$.
Let $P \in \cP$ correspond to the edge 
on this path incident with $w$ (which must exist as
$v$ is not a cut vertex), and let $P' \in \cP(B')$.
Then, by Lemma~\ref{lem:help}, $w(P)=B(P,P')$ and $w(P')=B(P',P)$.
Moreover, by Theorem~\ref{thm:bijection}, $w(P')=v(P')$
and $w(P) \neq v(P)$. Hence $v(P) \cup v(P') \neq X$,
which is a contradiction as $P$ and $P'$ are in distinct components of $G_{v}$.  
\end{proof}

%
%
%

\section{The block decomposition of a 
quasi-median graph}\label{sec:block_decomp}

As stated in the introduction, we want to determine the 
blocks of the quasi-median graph $Q(\cP)$ of a 
partition system $\cP$ without having to compute $Q(\cP)$ itself.
To do this, rather than computing the blocks of $Q(\cP)$ directly,
we shall compute some sets associated with 
each block which we now define. 

Given a block $B$ of $Q(\cP)$, we let 
$X(B)=V(B)\cap \pi(X)$ denote the set of 
vertices in $B$ labelled by elements in $X$, $\cP(B)$ the set
of partitions in $\cP$ corresponding to edges of $B$ and 
$S(B)$ the set of  cut
vertices of $Q(\cP)$ that are in $B$ but not in $X(B)$.
Note that $X(B)$ or $S(B)$ can be empty,
but that $X(B) \cup S(B)$ is never empty.
We will also consider 
the set $\cP_r(B)$ of partitions of
the set $X(B) \cup S(B)$ that is induced 
by, for each $P \in \cP(B)$, removing 
all those edges in $B$ that correspond to $P$.

\begin{exmp}\label{eg}
For  the large block $B$ in the middle of 
the quasi-median graph in Example~\ref{ex:one}, 
we have $X(B)=\{s_7,s_8\}$, 
$S(B)=\{e_1,e_4,e_5\}$, $\cP(B)=\{P_1,P_2,P_3,P_8\}$ 
and $\cP_r(B)=\{P'_1,P'_2,P'_3,P'_8\}$, where
\begin{align*}
P'_1&=\{\{s_7,s_8,e_5\},\{e_1,e_4\}\},&P'_2=\{\{e_1,e_5,s_8\},\{s_7,e_4\}\}\,,\\
P'_3&=\{\{s_7,s_8,e_1\},\{e_4,e_5\}\},&P'_8=\{\{s_7,e_1,e_4\},\{s_8,e_5\}\}\,.
\end{align*}
\end{exmp}

Now, we define the \emph{block decomposition} $\cB(\cP)$ of 
the quasi-median graph of a partition system $\cP$ on 
the set $X$ to be the set 
\[
\SetOf{ (X(B),S(B),\cP_r(B))}{B \mbox{ is a block of } Q(\cP) }.
\]
Our main aim is to compute this decomposition without
having to compute $Q(\cP)$.
Note that in view of the following lemma we can always
reconstruct $Q(\cP)$ from $\blist(\cP)$.

\begin{lem} 
Given a partition system $\cP$ and a block 
$B$ of $Q(\cP)$, the quasi-median graph 
$Q(\cP_r(B))$ is isomorphic to $B$.
\end{lem}
\begin{proof}
By definition, a $\cP$-map $v$ is a vertex of the 
block $B$ if and only if $v$ is contained in some 
edge of $Q(\cP)$ corresponding to an element of $\cP(B)$. 
Consider now the $\cP_r(B)$-map $v'$ that maps 
a partition $P'\in\cP_r(B)$ to that $A'\in P'$ that corresponds to the part $A=v(P)\in P$. This is a vertex of 
$Q(\cP_r(B))$ and it can be easily seen that the map 
$v\mapsto v'$ induces the desired 
isomorphism between $B$ and $Q(\cP_r(B))$. 
\end{proof}

\begin{rem}
In \cite[Theorem 3]{SD11}, Schwarz and D\"ur define
what they call the \emph{Block Decomposition of a Quasi-Median Network}. 
However, they do not use the notion of \emph{block} in the 
usual graph theoretical way. Instead, they work with a 
notion that is suitable for their aim of visualising 
quasi-median graphs. In particular, their blocks depend on 
an arbitrary vertex of the quasi-median graph which can 
be chosen in a suitable way to obtain 
improved visualisations.
\end{rem}

In what follows, we shall not directly compute 
the block decomposition of $Q(\cP)$, but instead
some closely related data from  
which the  block decomposition can be easily computed.

To this end, let $S(\cP)$ denote the union of all $S(B)$ 
with $B$ a block of $Q(\cP)$; we call any
element in $S(\cP)$ an \emph{extra vertex}. 
For $v\in S(\cP)$ we denote the set of all blocks $B$ 
in $Q(\cP)$ with $v \in S(B)$ by $B(v)$. 
An element $x\in X$ is \emph{in the direction of}
$B$ with respect to $v\in S(B)$ if every path from 
$x$ to $v$ has an edge in $B$.
Note that since all vertices of $Q(\cP)$
are elements of the quasi-median hull
of $\pi(X)$, there always exists such an 
element $x(v,B)$ although this element is not necessarily unique. 

\begin{lem}\label{lem:cpcpr}
Suppose that $\cP$ is a partition system on $X$
and $B$ is a block of $Q(\cP)$. If we are given the sets
$X(B)$, $S(B)$, $\cP(B)$ and, for each $v\in S(B)$ and some $C\in B(v) \setminus \{B\}$some element $x(v,C)$ in the direction of $C$ with respect to $v$,
then we can obtain the set $\cP_r(B)$ from the set $\cP(B)$
in time $\bigO(n m)$, where $n=\card X$, $m=\card \cP$.
\end{lem}

\begin{proof}
For each partition $P\in\cP(B)$ we 
construct a partition $P'$ of $X(B)\cup S(B)$ as 
follows. Elements of $X(B)$ are in that part of $P'$ 
that they are in $P$. For each $v \in S(B)$ we choose 
some $C \in B(v) \setminus \{B\}$ and put $v$ in that 
part of $P'$ that $x(v,C)$ is in $P$. Repeating this 
for all partitions $P\in\cP(B)$ gives us the set $\cP_r(B)$.
This procedure can be carried out in time $O(mn)$, giving
the desired run time bound.
\end{proof}

\begin{exmp}
To compute $\cP_r(B)$ from $\cP(B)$ and the 
information $x(v,B)$ for all $v\in S(B)$
for the block $B$ in Example~\ref{eg}, assume 
that $x(e_1,B_7)=s_3,x(e_4,B_4)=s_5$ and 
$x(e_5,B_{10})=s_1$, where, for this moment, we 
denote by $B_i$ the block containing 
the (sole) partition $P_i$.

Now, we start out with partition $P_1$ and 
have to check in which part of the partition the 
extra points $e_1,e_4$ and $e_5$ are contained. Since 
$x(e_1,B_7)=s_3$, we substitute $s_3$ for $e_1$ 
in $P_1$ and, similarly, we substitute $s_5$ 
for $e_4$ and $s_1$ for $e_5$. Deleting all 
$x\in X\setminus X(B)$ in the remaining 
partition yields the partition $P'_1$. After 
performing the same process for $P_2,P_3$ and $P_8$, 
we obtain the set $\cP_r(B)$.
\end{exmp}

So, to compute the block decomposition of the 
quasi-median graph of a partition system $\cP$
it suffices to compute, for each block $B$ of $Q(\cP)$,
the sets $X(B)$, $S(B)$ and $\cP(B)$, 
and also, for  each $v \in S(B)$ and $B\in B(v)$, 
some element $x(v,B)$ in the direction of $B$ with respect to $v$.
In the next section we shall present an 
algorithm for doing precisely this.

\begin{algorithm}[htp]
\label{alg:add:partition}
 \KwIn{The set 
$\blist = \{(X(B),S(B),\cP(B)) \,:\, B \mbox{ a block of } Q(\cP) \}$ for a partition system $\cP$
and, for each $v \in S(B)$ and $B\in B(v)$, 
some element $x(v,B)$ in the direction of $B$ with respect to $v$,
together with some partition $P\not\in\cP$.}
 \KwOut{The same data for $\cP\cup\{P\}$.}
  Create a new block $C$ with $X(C)=X$, $S(C)=\emptyset$, $\cP(C)=\{P\}$\;
  Create a new extra vertex $v$\;
  $\IB\leftarrow\emptyset$\;
  \ForEach{$B\in\blist$} {
    \If{!\FuncSty{is\_compatible}$(P,B)$} {
      Add $B$ to $\IB$\;
    }
    \Else{
      Choose some $Q\in\cP(B)$\;
   $X(B)\leftarrow X(B)\cap B(P,Q)$\;
    \If{$X(C)\cap X(B)=\emptyset$} {\label{line:add:extra-point}       
    Choose some $x\in B(P,Q)$ and some $y\in B(Q,P)$\;
    \If{There exists some $v\in S(C)$ such that $x(v,B)$ and $x$ are in the same part of $P$} {
     \label{line:find:extra-point}         $w\leftarrow v$\;
         }
         \ElseIf{There exists some $v\in S(B)$ such that $x(v,C)$ and $y$ are in the same part of $Q$} {
     \label{line:find:extra-point2}         $w\leftarrow v$\;
         }
         \Else {
           $w\leftarrow$new extra vertex\;
         }
         $x(w,B)\leftarrow$\FuncSty{add\_extra\_vertex}($w$,$B$)\;
         $x(w,C)\leftarrow$\FuncSty{add\_extra\_vertex}($w$,$C$)\;
      }
    }
  }
  \If{$\IB\not=\emptyset$} {
     \FuncSty{add\_blocks}$(C,\IB)$\;
  }
\KwRet{$\blist \cup \{(X(C),S(C),\cP(C))\}$ \emph{and the elements} $x(w,C)$ }\;
\caption{Algorithm to add a partition.}
\end{algorithm}

\section{Computing the block decomposition of
a quasi-median graph}\label{sec:compute_block}

We now present our approach to computing the
block decomposition of a partition system $\cP$ following the
strategy presented at the end of the last section.
We start with the block decomposition 
of an empty set of partitions on $X$ (which is 
itself empty) and iteratively 
add each $P \in \cP$ to build up the decomposition.
In particular, at each stage, for each 
block $B$ (either existing or new) we compute
the sets $X(B)$, $S(B)$, $\cP(B)$, together with  
elements $x(v,B)$, $v \in S(B)$, $B\in B(v)$.  
To do this we use Algorithm~\ref{alg:add:partition},
the main elements for which are as follows. 

First, for each given block $B$, we check whether or not
there exists some partition in $\cP(B)$ that is not strongly 
compatible to the newly added partition $\cP$ and 
thereby also compute which elements of $X$ must 
be added to our new block. This is done in the 
function \texttt{is\_compatible} described in 
Algorithm~\ref{alg:is_compatible}. This function 
returns \TRUE~if the new partition $P$ is strongly
compatible to all partitions $Q$ in the block $B$. 
All blocks $B$ with \texttt{is\_compatible}$(P,B)$=\TRUE~remain 
blocks for the new block decomposition, and all other blocks 
are joined (together with $P$) to form a new block that is 
added to the decomposition. This is done in the 
function \texttt{join\_blocks} outlined 
in Algorithm~\ref{alg:join_blocks}.

We now prove that this approach really works:

\begin{thm}\label{thm:correct}
Algorithm~\ref{alg:add:partition} is correct.
\end{thm}
\begin{proof}
We first show that if the sets $\cP(B)$ and $X(B)$ 
have been computed correctly for all blocks $B$ of
the quasi-median graph of the
partition system $\cP \setminus\{P\}$, then they are
correct for all blocks $Q(\cP)$. 

To see that all $\cP(B)$ are correct, 
note that the set $\cP(C)$ for the new block $C$ is 
initialised as $\{P\}$ and in the function \texttt{add\_blocks} 
all partitions of blocks containing partitions not strongly compatible 
to $P$ are added and the corresponding blocks deleted. 
Hence, it follows from Corollary~\ref{cor:partitions}
that $\cP(B)$ is correct for all blocks of $Q(\cP)$.

%

We now turn to the correctness of the set
$X(B)$. Consider first a block $B$ for which 
every partition $Q\in \cP(B)$ is strongly compatible 
with $P$. The elements of $X(B)$ stay in $X(B)$ if they are 
in $B(P,Q)$, and similarly move to $X(C)$ 
for the new block $C$ if they are 
in $B(Q,P)$. But, by Theorem~\ref{thm:non-compatible}~(i),
the quasi-median graph $Q(\{P,Q\})$ 
has two blocks $B_1,B_2$ with $\cP(B_1)=\{P\}$, $X(B_1)=B(P,Q)$ 
and $\cP(B_2)=\{Q\}$, $X(B_2)=B(Q,P)$.
It follows that $X(B)$ is correct. 
Otherwise, if some $Q\in \cP(B)$ is not strongly
compatible to $P$, then the corresponding block is 
deleted and all elements are simply joined to those 
in $X(C)$, as required. So, using a similar 
argument for $Q(\{P,Q\})$, it follows 
that $X(C)$ is also correct.

It remains to show that the blocks are 
added in a proper way, that is, all of the 
extra vertices are contained in the blocks that they 
really belong to. This is taken care of by the 
condition in Line~\ref{line:add:extra-point} of 
Algorithm~\ref{alg:add:partition}: There is no 
need to add extra vertices for adding two 
blocks if they already share an element of $X$ 
and having $X(B)\not\subseteq B(P,Q)$ ensures 
that blocks are only added if needed. Moreover, 
Algorithm~\ref{alg:connect_extra_point} 
ensures that elements in the direction of some 
block are computed.  Indeed,
suppose all existing $x(\cdot,\cdot)$ are correct. 
To see that Algorithm~\ref{alg:connect_extra_point} returns 
an element of $x$ that is in the direction of $B$
first note that if $x\in X(B)$ and $X(B)\not=\emptyset$, then $x$ 
is clearly in the direction of $B$ with respect to $v$. 
Furthermore, every $w\in S(B)\setminus\{v\}$ is in the 
direction of $B$ with respect to $v$ and so every 
element in the direction of any $C\in B(w)\setminus\{B\}$ 
with respect to $w$ is in the direction of $B$ with respect to $v$.
This completes the proof of the theorem.
\end{proof}

\begin{algorithm}[ht]
\label{alg:is_compatible}
\FuncSty{is\_compatible}$(P,B)$\\
\ForEach{Partition $Q\in \cP(B)$} {
   \If{$P$ and $Q$ are strongly compatible} {
      $X(C)\leftarrow X(C)\cap B(Q,P)$\;
   }
   \Else{
      \KwRet{$\FALSE$}\;    
   }
}
\KwRet{$\TRUE$}\;
\caption{Check if a partition is strongly compatible with all 
partitions arising from a block.}
\end{algorithm}

\begin{algorithm}[ht]
\label{alg:join_blocks}
     \FuncSty{add\_blocks}$(C,\IB)$\\
     $X(C)\leftarrow\emptyset$\;
     \ForEach{$B\in \IB$} {
       Remove $B$ from $\blist$\;
       $X(C)\leftarrow X(C)\cup X(B)$\;
       $\cP(C)\leftarrow \cP(C)\cup \cP(B)$\;          
       \ForEach{$w\in S(B)$} {
         Add $w$ to $S(C)$\;
         $x(w,C)\leftarrow x(w,B)$\;
       }
     }
 \ForEach{$w\in S(C)$} {
       \If{$B(w)\subseteq \IB\cup \{C\}$} {
          Delete the extra vertex $w$ from $S(C)$\;
       }
     }
\caption{Add all blocks not strongly compatible to $P$.}
\end{algorithm}

\begin{algorithm}
\label{alg:connect_extra_point}
  \FuncSty{add\_extra\_vertex}($v$,$B$)\\
  Add $v$ to $S(B)$\;
  \If{ $X(B)\not=\emptyset$} {
     Choose some $x\in X(B)$\;
     \KwRet{$x$}\;
  }
  Choose some $w\in S(B)\setminus\{v\}$\;
  Choose some $C\in B(w)\setminus\{B\}$\; 
  \KwRet{$x(w,C)$}\;
\caption{Add an extra vertex to a block.}
\end{algorithm}

We conclude with an analysis of 
the run time of Algorithm~\ref{alg:add:partition}. 
First, we compute the time needed to check whether two partitions are strongly compatible.

\begin{lem}\label{lem:check_comp}
Let $P$ and $Q$ be partitions of $X$ with $\card X=n$. Then  checking strong 
compatibility and computing $B(P,Q)$ and $B(Q,P)$ in case 
they are strongly compatible can be done in time $\bigO(n)$.
\end{lem}
\begin{proof}
We can rename the elements of $X$ in such a way that $X=\{s_1,\dots,s_n\}$ and the elements of $P$ are all intervals of the sequence $s_1,\dots,s_n$, that is, of the form $\{s_i,s_{i+1},\dots,s_{j-1},s_j\}$ for some $1\leq i\leq j\leq n$. This relabelling can be done in time linear in $n$. Next, we fix some order $\phi$ (that is a bijection $\phi: Q \to \card Q$) on $Q$ and define a sequence $S_Q$ of length $n$ where the $i$th element of the sequence is $\phi(A)$, if $i\in A\in Q$.  By going through the elements of $Q$, we can construct this sequence in time linear in $n$ and independent of $\card Q$.

By the construction of this sequence, we have that for any element $A=\{s_i,s_{i+1},$ $\dots,s_{j-1},s_j\}\in P$, there exists some $B\in Q$ with $A\cup B=X$ if and only if  all $\alpha<i$ and $\alpha>j$ have the same value in the sequence $S_Q$, that is, if $S_Q$ has a constant prefix of length at least $i-1$ and a constant suffix of length at least $n-j$ and those two have the same value. Hence, to check whether $P$ and $Q$ are strongly compatible, it now suffices to compute the maximum length constant prefixes and suffixes of $S_Q$ and then check for each $A\in P$ whether the above condition is fulfilled; both can be done in time linear in $n$. In case one $A\in P$ fulfilling the condition is found, we also know that $B(P,Q)=A$ and $B(Q,P)=\phi^{-1}(c)$ where $c$ is the constant of the prefix/suffix of $S_Q$
\end{proof}

\begin{thm} \label{thm:run_time}
The algorithm computes the block decomposition of a partition 
system $\cP$ on $X$ in time $\bigO(n^2m^2)$, where 
$n=\card X$ and  $m=\card \cP$.
\end{thm}
\begin{proof}
We claim that  Algorithm~\ref{alg:add:partition} 
runs in time $\bigO(n^2m)$. Since this algorithm is 
executed once for each partition, the 
theorem then follows by Lemma~\ref{lem:cpcpr} and Corollary~\ref{prop:num-cut vertices}.

It follows from Lemma~\ref{lem:check_comp} that the  
function \texttt{is\_compatible} in 
Algorithm~\ref{alg:is_compatible} runs in time 
$\bigO(n\cdot \card{\cP(B)})$. The rest of the 
first loop in Algorithm~\ref{alg:add:partition} is 
dominated by the conditions in Lines~\ref{line:find:extra-point} and~\ref{line:find:extra-point2}. 
However, since the number of extra vertices of $Q(\cP)$ is 
linear in $n$ by Proposition~\ref{prop:num-cut vertices}, 
this test can be performed in $\bigO(n^2)$. Since each 
partition can only be in one block, this shows that the 
loop in Algorithm~\ref{alg:add:partition} 
needs $\bigO((n+n^2)m)=\bigO(n^2m)$ time. For the 
function \texttt{add\_blocks} the run time of the 
first loop is bound by $\bigO(n^2)$, taking into account 
that by Proposition~\ref{prop:num-cut vertices} the number 
of extra vertices and the number of blocks are linear in $n$. 
The same holds for the second loop, so \texttt{add\_blocks} 
runs in time $\bigO(n^2)$. Altogether, we get that
Algorithm~\ref{alg:add:partition} runs in 
time $\bigO(n^2m)$, as claimed. 
\end{proof}

Note that, translated into the language of sequences used in the
introduction, this results implies that the
block decomposition of the quasi-median graph 
of $n$ aligned sequences of length $m$
can be computed in time $\bigO(n^2m^2)$,

\bibliographystyle{amsplain}
\bibliography{quasi-median}

\end{document}

%% file: head.tex
\usepackage{amsfonts, amsmath, amssymb, amsthm}
\usepackage{german}
\usepackage[ngerman,english]{babel}
\usepackage{setspace}

\theoremstyle{plain}
\newtheorem{thm}{Theorem}[section]
\newtheorem*{thm*}{Theorem}
\newtheorem{lem}[thm]{Lemma}
\newtheorem{prop}[thm]{Proposition}
\newtheorem{cor}[thm]{Corollary}

\usepackage[linesnumbered,ruled,norelsize]{algorithm2e}

\theoremstyle{definition}

\newtheorem{exmp}[thm]{Example}

\theoremstyle{remark}
\newtheorem{rem}[thm]{Remark}

\usepackage{txfonts}
\usepackage{amssymb}
\usepackage{mathbbol}
\usepackage{ae}
\usepackage{bbm}
\usepackage[mathscr]{eucal} 
\usepackage{xspace}
\usepackage{url}

\usepackage{color}
\usepackage{graphicx}
\usepackage{overpic}
\usepackage{subfigure}
\usepackage{booktabs}
\usepackage{paralist}

\usepackage[a4paper,scale=0.7, marginratio={1:1, 9:10}, ignoreall]{geometry}

\numberwithin{equation}{section}
\numberwithin{figure}{section}

\newcommand\cP{{\mathcal P}}

\newcommand\cB{{\mathcal B}}

\newcommand\cR{{\mathcal R}}

\newcommand\SetOf[2]{\left\{#1\vphantom{#2}\,\right.\left|\,\vphantom{#1}#2\right\}}
\newcommand\smallSetOf[2]{\{#1\,|\,#2\}}

\newcommand\card[1]{\left|#1\right|}

\renewcommand{\phi}{\varphi}

\newcommand{\IB}{{\mathcal{B}_{\textrm{incomp}}}}
\newcommand{\blist}{\mathcal B}

\newcommand{\TRUE}{\texttt{TRUE}}
\newcommand{\FALSE}{\texttt{FALSE}}
\newcommand{\Part}{\operatorname{Part}}
\newcommand{\bigO}{\mathcal O}

\usepackage[T1]{fontenc}
\usepackage[utf8]{inputenc}
